\documentclass[12pt,reqno]{amsart}
\usepackage{amsmath,amssymb,amsfonts,amscd,latexsym,amsthm,mathrsfs,verbatim,comment}
\usepackage[unicode]{hyperref}
\newtheorem{lemma}{Lemma}
\newtheorem{theorem}{Theorem}

\newtheorem{proposition}{Proposition}
\theoremstyle{remark}
\newtheorem{remark}{\bf Remark}

\let\wh\widehat
\let\wt\widetilde
\newcommand\m{\mathrm{m}}

\renewcommand{\d}{{\mathrm d}}

\renewcommand{\Im}{\operatorname{Im}}
\renewcommand{\Re}{\operatorname{Re}}
\newcommand{\sign}{\operatorname{sign}}
\begin{document}

\title{On the Mahler measure of hyperelliptic families}

\author{Marie Jos\'e Bertin}
\address{Universit\'e Pierre et Marie Curie (Paris 6), Institut de Math\'ematiques, 4 Place Jussieu, F-75252 Paris, FRANCE}
\email{marie-jose.bertin@imj-prg.fr}

\author{Wadim Zudilin}
\address{School of Mathematical and Physical Sciences, The University of Newcastle, Callaghan NSW 2308, AUSTRALIA}
\email{wzudilin@gmail.com}

\date{28 January 2016}

\begin{abstract}
We prove Boyd's ``unexpected coincidence'' of the Mahler measures for two families of two-variate polynomials defining curves of genus~2.
We further equate the same measures to the Mahler measures of polynomials $y^3-y+x^3-x+kxy$ whose zero loci define elliptic curves for $k\ne0,\pm3$.
\end{abstract}

\subjclass[2010]{Primary 11F67; Secondary 11F11, 11F20, 11G16, 11G55, 11R06, 14H52, 19F27}
\keywords{Mahler measure, $L$-value, elliptic curve, hyperelliptic curve, elliptic integral}


\maketitle

\setcounter{section}{-1}

\section{Introduction}
\label{sintro}

In his pioneering systematic study \cite{Boy98} of the Mahler measures of two-variate polynomials
D.~Boyd has distinguished several special families, for which the measures are related to the $L$-values of the curves defined by the zero loci of the polynomials.
The two particular families
$$
P_k(x,y)
=(x^2+x+1)y^2+kx(x+1)y+x(x^2+x+1)
$$
and
$$
Q_k(x,y)
=(x^2+x+1)y^2
+(x^4+kx^3+(2k-4)x^2+kx+1)y
+x^2(x^2+x+1)
$$
are nicknamed in \cite{Boy98} as Family~3.2 and Family~3.5B, respectively.
Generically, both $P_k(x,y)=0$ and $Q_k(x,y)=0$ define curves of genus~2
whose jacobians are isogenous to the product of two elliptic curves.
Computing the Mahler measures of $P_k(x,y)$ and $Q_k(x,y)$ numerically and identifying them as rational multiples of the $L$-values
$L'(E_k,0)$, where
\begin{equation}
E_k:y^2=x^3+(k^2-24)x^2-16(k^2-9)x
\label{Ek}
\end{equation}
is isomorphic to one of the elliptic curves in the product for each of the two families,
Boyd observes the ``unexpected coincidence'' $\m(P_k)=\m(Q_{k+2})$ for integer $k$
in the range $4\le k\le33$ (but not for $k\le3$).
The primary goal of this note is to confirm Boyd's observation.

\begin{theorem}
\label{th1}
For real $k\ge4$, we have $\m(P_k)=\m(Q_{k+2})$.
\end{theorem}

Note that for $k\ne0,\pm3$ the curve $E_k$ is elliptic and it is isomorphic to the elliptic curve $R_k(x,y)=0$, where the polynomial
$$
R_k(x,y)=y^3-y+x^3-x+kxy
$$
is tempered\,---\,all the faces of its Newton polygon are represented by cyclotomic polynomials.
The elliptic origin of the family $R_k(x,y)$ and Beilinson's conjectures predict \cite{Boy98,RV99} that, apart from a finite set of $k$,
the measure $\m(R_k)$ is $\mathbb Q$-proportional to the $L$-value $L'(E_k,0)$ for $k\in\mathbb Z$ (in fact, even for $k$ such that $k^2\in\mathbb Z$
as in any such case the curve $R_k(x,y)=0$ possesses the model defined over~$\mathbb Z$).
Our next result unites the predictions with the findings of Boyd in~\cite{Boy98}.

\begin{theorem}
\label{th2}
For real $k$ satisfying $|k|\ge16/(3\sqrt3)=3.0792\dots$, we have $\m(P_k)=\m(R_k)$.
\end{theorem}

Noticing that $P_{-k}(x,y)=P_k(x,-y)$ and $R_{-k}(x,y)=R_k(-x,-y)$ we conclude that $\m(P_{|k|})=\m(P_k)$ and $\m(R_{|k|})=\m(R_k)$,
hence it is sufficient to establish the identity in Theorem~\ref{th2} and analyse the two polynomial families for positive real $k$ only.

Our analysis of the three polynomial families is performed in Sections~\ref{sP}--\ref{sR}, each section devoted to one family.
We compute the derivatives of the corresponding Mahler measures with respect to the parameter $k$ and make use
of the easily seen asymptotics
\begin{equation}
\m(P_k)=\log|k|+o(1), \quad
\m(Q_k)=\log|k|+o(1) \quad\text{and}\quad
\m(R_k)=\log|k|+o(1)
\label{asymp}
\end{equation}
as $|k|\to\infty$, to conclude about the equality of the Mahler measures themselves. This is a strategy we have successfully employed before in \cite{BZ14}.
Our findings provide one with the reasons of why the ranges for $k$ in Theorems~\ref{th1} and~\ref{th2} cannot be refined, and in Section~\ref{sfinale}
we discuss some further aspects of this ``expected noncoincidence.''

One of our reasons for linking the Mahler measures of hyperelliptic families $P_k(x,y)$ and $Q_k(x,y)$
to that of elliptic family $R_k(x,y)$, not previously displayed, is a hope to actually prove
$\m(R_k)=c_kL'(E_k,0)$ with $c_k\in\mathbb Q^\times$ for some values of~$k$. Armed with the recent formula
for the regulator of modular units \cite{Zud14} and its far-going generalisation for the regulator of Siegel units \cite{Bru15}
established by F.~Brunault, such identities are expected to be automated in the near future.
The main obstacle to produce a single example for $\m(R_k)$ is of purely computational nature: the smallest conductor
of the elliptic curve $E_k$ one gets for $k>3$, $k^2\in\mathbb Z$, is $224=2^5\times7$ when $k=4$.
We further comment on this circumstance and on a related conjecture of Boyd for $\m(Q_{-1})$ in the final section.

\section{The first family}
\label{sP}

We use the equality $\m(P_{|k|})=\m(P_k)$ to reduce our analysis in this section to that for $k\ge0$.

Write $P_k(x^2,y)=x^4\wt P_k(x,y/x)$, where
\begin{align*}
\wt P_k(x,y)
&=(x^2+x^{-2}+1)y^2+k(x+x^{-1})y+(x^2+x^{-2}+1)
\\
&=(x+x^{-1}+1)(x+x^{-1}-1)y^2+k(x+x^{-1})y+(x+x^{-1}+1)(x+x^{-1}-1)
\\
&=(x+x^{-1}+1)(x+x^{-1}-1)(y-y_1(x))(y-y_2(x))
\end{align*}
and
\begin{gather*}
\{y_1(x),y_2(x)\}=\frac{-k(x+x^{-1})\pm\sqrt{\Delta_k(x)}}{2(x+x^{-1}+1)(x+x^{-1}-1)}
\\
\Delta_k(x)=k^2(x+x^{-1})^2-4((x+x^{-1})^2-1)^2.
\end{gather*}
By Vi\`ete's theorem $y_1(x)y_2(x)=1$ implying that $|y_1(x)|=|y_2(x)|=1$ if $\Delta_k(x)\le0$ and
$|y_2(x)|<1<|y_1(x)|$ if $\Delta_k(x)>0$, when we order the zeroes $y_1(x),y_2(x)$ appropriately.
In the latter case
$$
|y_1(x)|=\max\{|y_1(x)|,|y_2(x)|\}=\frac{k|x+x^{-1}|+\sqrt{\Delta_k(x)}}{2|(x+x^{-1})^2-1|}>1
$$
and
$$
|y_2(x)|=\min\{|y_1(x)|,|y_2(x)|\}<1.
$$

In notation $x=e^{i\theta}$, $-\pi<\theta<\pi$, we let $c=\cos^2\theta$, so that $c$ ranges in $[0,1]$.
Since $x+x^{-1}=2\cos\theta$, we get
\begin{align*}
\Delta_k
&=4k^2c-4(4c-1)^2
=-4(16c^2-(8+k^2)c+1)
\\
&=-64(c-c_-(k))(c-c_+(k)),
\end{align*}
where
$$
c_{\pm}(k)=\frac{8+k^2\pm k\sqrt{16+k^2}}{32}.
$$
Because $0<c_-(k)<c_+(k)<1$ for $0<k<3$ and $0<c_-(k)<1<c_+(k)$ if $k>3$, we have $\Delta_k\ge0$ iff
$c_-(k)\le c\le\min\{1,c_+(k)\}$. Note that
$$
|y_1(x)|=\frac{k\sqrt{c}+4\sqrt{-(c-c_-(k))(c-c_+(k))}}{|4c^2-1|}.
$$

Using Jensen's formula and the symmetry $y_1(x)=y_1(x^{-1})$, we obtain
\begin{align*}
p(k)
&=\m(P_k(x,y))=\m(\wt P_k(x,y))
\\
&=\frac1{(2\pi i)^2}\iint_{|x|=|y|=1}\log|\wt P_k(x,y)|\,\frac{\d x}x\,\frac{\d y}y
\displaybreak[2]\\
&=\frac1{2\pi i}\int_{|x|=1}\log|y_1(x)|\,\frac{\d x}x
\displaybreak[2]\\
&=\frac1{\pi i}\int_{\substack{|x|=1\\\Im x>0}}\Re\log y_1(x)\,\frac{\d x}x
\displaybreak[2]\\
&=\frac1{\pi i}\int_{\substack{|x|=1\\\Im x>0}}\Re\log\frac{k|x+x^{-1}|+\sqrt{\Delta_k(x)}}{2(x+x^{-1}+1)(x+x^{-1}-1)}\,\frac{\d x}x
\displaybreak[2]\\
&=\frac1{\pi i}\int_{\substack{|x|=1\\\Im x>0}}\Re\log\frac{k|x+x^{-1}|+\sqrt{\Delta_k(x)}}{2}\,\frac{\d x}x
\\
&=\frac1\pi\,\Re\int_0^\pi\log\Bigl(k|\cos\theta|+\sqrt{-(16\cos^4\theta-(8+k^2)\cos^2\theta+1)}\Bigr)\,\d\theta,
\quad k>0.
\end{align*}
The derivative of the result with respect to $k$ is
\begin{align*}
\frac{\d p(k)}{\d k}
&=\frac1\pi\,\Re\int_0^\pi\frac{|\cos\theta|}{\sqrt{-(16\cos^4\theta-(8+k^2)\cos^2\theta+1)}}\,\d\theta
\\
&=\frac1\pi\,\Re\int_{-1}^1\frac{|t|}{\sqrt{-(16t^4-(8+k^2)t^2+1)}}\,\frac{\d t}{\sqrt{1-t^2}}
\\
&=\frac2\pi\,\Re\int_0^1\frac{t}{\sqrt{-(16t^4-(8+k^2)t^2+1)}}\,\frac{\d t}{\sqrt{1-t^2}}
\\
&=\frac1\pi\,\Re\int_0^1\frac{1}{\sqrt{-(16c^2-(8+k^2)c+1)}}\,\frac{\d c}{\sqrt{1-c}}
\\
&=\frac1{4\pi}\int_{c_-(k)}^{\min\{1,c_+(k)\}}\frac{\d c}{\sqrt{(c-c_-(k))(c-c_+(k))(c-1)}},
\end{align*}
which is a complete elliptic integral.

Performing additionally the change $c=(4-v)/16$ we obtain
\begin{align*}
\frac{\d p(k)}{\d k}
&=\frac1\pi\,\Re\int_{-12}^4\frac{\d v}{\sqrt{-(v+12)(v^2+k^2v-4k^2)}}
\\
&=\frac1\pi\int_{\max\{-12,-k(k+\sqrt{k^2+16})/2\}}^{-k(k-\sqrt{k^2+16})/2}\frac{\d v}{\sqrt{-(v+12)(v^2+k^2v-4k^2)}};
\end{align*}
in particular, we have the following.

\begin{proposition}
For $k\ge3$,
\begin{equation}
\frac{\d p(k)}{\d k}
=\frac1\pi\int_{-12}^{-k(k-\sqrt{k^2+16})/2}\frac{\d v}{\sqrt{-(v+12)(v^2+k^2v-4k^2)}}.
\label{p(k)}
\end{equation}
\end{proposition}

\section{The second family}
\label{sQ}

The analysis here is very similar to the one we had in the paper~\cite{BZ14}.
First introduce $Q_{k+2}(x,y)=x^3\wt Q_{k+2}(x,y/x)$, where
\begin{align*}
\wt Q_{k+2}(x,y)
&=(x+x^{-1}+1)y^2+(x^2+x^{-2}+(k+2)(x+x^{-1})+2k)y+(x+x^{-1}+1)
\\
&=(x+x^{-1}+1)y^2+\bigl((x+x^{-1})^2+(k+2)(x+x^{-1})+2(k-1)\bigr)y
\\ &\qquad
+(x+x^{-1}+1).
\end{align*}
Write
$$
\wt Q_{k+2}(x,y)=(x+x^{-1}+1)(y-y_1(x))(y-y_2(x)),
$$
where
$$
\{y_1(x),y_2(x)\}=\frac{-B_k(x)\pm\sqrt{\Delta_k(x)}}{2(x+x^{-1}+1)}
$$
and $B_k(x)=(x+x^{-1})^2+(k+2)(x+x^{-1})+2(k-1)$,
\begin{align*}
\Delta_k(x)
&=B_k(x)^2-4(x+x^{-1}+1)^2
\\
&=(x+x^{-1}+2)(x+x^{-1}+k-2)((x+x^{-1})^2+(k+4)(x+x^{-1})+2k).
\end{align*}
By Vi\`ete's theorem $y_1(x)y_2(x)=1$ implying that $|y_1(x)|=|y_2(x)|=1$ if $\Delta_k(x)\le0$ and
$|y_2(x)|<1<|y_1(x)|$ if $\Delta_k(x)>0$, when we order the zeroes $y_1(x),y_2(x)$ appropriately.
In the latter case
$$
y_1(x)=\frac{-B_k(x)-\sign(B_k(x))\sqrt{\Delta_k(x)}}{2(x+x^{-1}+1)}.
$$
Note that
\begin{align*}
\frac{\d}{\d k}\,\log y_1(x)
&=\frac{\d}{\d k}\,\log\Bigl(B_k(x)+\sign(B_k(x))\sqrt{B_k(x)^2-4(x+x^{-1}+1)^2}\Bigr)
\\
&=\frac{\d}{\d B}\,\log\Bigl(B+\sign(B)\sqrt{B^2-4(x+x^{-1}+1)^2}\Bigr)\bigg|_{B=B_k(x)}\cdot\frac{\d B_k}{\d k}
\\
&=-\frac{\sign(B_k(x))}{\sqrt{B_k(x)^2-4(x+x^{-1}+1)^2}}\cdot(x+x^{-1}+2).
\end{align*}
With the help of Jensen's formula we obtain
\begin{align*}
q(k+2)
&=\m(Q_{k+2}(x,y))=\m(\wt Q_{k+2}(x,y))
\\
&=\frac1{(2\pi i)^2}\iint_{|x|=|y|=1}\log|\wt Q_k(x,y)|\,\frac{\d x}x\,\frac{\d y}y
\displaybreak[2]\\
&=\frac1{2\pi i}\int_{|x|=1}\log|y_1(x)|\,\frac{\d x}x
\displaybreak[2]\\
&=\frac1{\pi i}\int_{\substack{|x|=1\\\Im x>0}}\Re\log y_1(x)\,\frac{\d x}x
\\
&=\frac1\pi\,\Re\int_0^\pi\log y_1(e^{i\theta})\,\d\theta,
\end{align*}
leading to
\begin{align*}
\frac{\d q(k+2)}{\d k}
&=-\frac1\pi\,\Re\int_0^\pi
\frac{\sign(B_k(e^{i\theta}))}{\sqrt{\Delta_k(e^{i\theta})}}\,(2\cos\theta+2)\,\d\theta
\\
&=-\frac1\pi\,\Re\int_{-1}^1
\frac{\sign(2t^2+(k+2)t+k-1)}{\sqrt{4(t+1)(2t+k-2)(2t^2+(k+4)t+k)}}\,\frac{(2t+2)\,\d t}{\sqrt{1-t^2}}
\\
&=-\frac1\pi\,\Re\int_{-1}^1
\frac{\sign((t+1)(2t+k)-1)}{\sqrt{(1-t)(2t+k-2)(2t^2+(k+4)t+k)}}\,\d t.
\end{align*}
Note that for $k>0$ we have
\begin{alignat*}{2}
\Re\int_{-1}^1\sign(2t^2+(k+2)t+k-1)&=-\int_{-1}^{(-k-4+\sqrt{16+k^2})/4}+\int_{1-k/2}^1 &\qquad\text{if}\quad& 0<k\le3,
\\
&=-\int_{-1}^{1-k/2}+\int_{(-k-4+\sqrt{16+k^2})/4}^1 &\qquad\text{if}\quad& 3<k<4,
\\
&=\int_{(-k-4+\sqrt{16+k^2})/4}^1 &\qquad\text{if}\quad& k\ge4.
\end{alignat*}
Performing the change of variable $t=(v+2k(k+1))/(v-4k)$ we then obtain
$$
\frac{\d q(k+2)}{\d k}
=\frac1\pi\biggl(\int_{-\infty}^{-12}-\int_{-k(k+\sqrt{16+k^2})/2}^{k(1-k)}\biggr)\frac{\d v}{\sqrt{-(v+12)(v^2+k^2v-4k^2)}}
$$
if $0<k\le3$,
$$
\frac{\d q(k+2)}{\d k}
=\frac1\pi\biggl(\int_{-\infty}^{-k(k+\sqrt{16+k^2})/2}-\int_{-12}^{k(1-k)}\biggr)\frac{\d v}{\sqrt{-(v+12)(v^2+k^2v-4k^2)}}
$$
if $3<k<4$, and
\begin{equation}
\frac{\d q(k+2)}{\d k}
=\frac1\pi\int_{-\infty}^{-k(k+\sqrt{16+k^2})/2}\frac{\d v}{\sqrt{-(v+12)(v^2+k^2v-4k^2)}}
\label{q(k)}
\end{equation}
if $k\ge4$.

\begin{remark}
\label{rem1}
The appearance of incomplete elliptic integrals
$$
\int_{-k(k+\sqrt{16+k^2})/2}^{k(1-k)}\frac{\d v}{\sqrt{-(v+12)(v^2+k^2v-4k^2)}}
$$
and
$$
\int_{-12}^{k(1-k)}\frac{\d v}{\sqrt{-(v+12)(v^2+k^2v-4k^2)}}
$$
for $k<4$ hints on why the Mahler measures $q(k+2)$ are possibly not related to the corresponding $L$-values
(see the question marks and the ``half-Mahler'' measures $\m'$ in \cite[Table~9]{Boy98}).
Our next statement refers to the situation when incomplete elliptic integrals do not occur.
\end{remark}

\begin{proposition}
\label{p=q}
For $k\ge4$,
$$
\frac{\d p(k)}{\d k}
=\frac{\d q(k+2)}{\d k}.
$$
\end{proposition}

\begin{proof}
We will show that
\begin{multline}
\int_{-12}^{-k(k-\sqrt{16+k^2})/2}\frac{\d v}{\sqrt{-(v+12)(v^2+k^2v-4k^2)}}
\\
=\int_{-\infty}^{-k(k+\sqrt{16+k^2})/2}\frac{\d v}{\sqrt{-(v+12)(v^2+k^2v-4k^2)}}
\label{intpq}
\end{multline}
for $k>3$. On comparing the integrals in \eqref{p(k)} and~\eqref{q(k)} this implies the required coincidence.

The involution
$$
v\mapsto-\frac{4(3v+4k^2)}{v+12}
$$
interchanges $\infty$ with $-12$ and $-k(k+\sqrt{k^2+16})/2$ with $-k(k-\sqrt{k^2+16})/2$.
Applying the change to one of the integrals in \eqref{intpq} we arrive at the other.
\end{proof}

\begin{proof}[Proof of Theorem~\textup{\ref{th1}}]
Proposition \ref{p=q} implies that $p(k)=q(k+2)+C$ for $k\ge4$, with some constant $C$ independent of~$k$.
On using the asymptotics \eqref{asymp} we conclude that $C=0$, and the theorem follows.
\end{proof}

\section{The third family}
\label{sR}

Since $\m(R_{|k|})=\m(R_k)$, we assume that $k\ge0$ throughout the section.

For the elliptic family we write
$$
-y^3R_k\bigl(x/y,1/(xy)\bigr)
=\wt R_k(x,y)=(x+x^{-1})y^2-ky-(x^3+x^{-3}).
$$
This time the zeroes $y_1(x)$ and $y_2(x)$ of the quadratic polynomial $\wt R_k(x,y)$
satisfy
\begin{equation*}
y_1(x)y_2(x)=-\frac{x^3+x^{-3}}{x+x^{-1}}=-(x^2-1+x^{-2})=3-4\cos^2\theta.
\end{equation*}
We have
\begin{align*}
y_1(x)&=\frac{k+\sqrt{k^2-16\cos^2\theta\,(3-4\cos^2\theta)}}{4\cos\theta},
\\
y_2(x)&=\frac{k-\sqrt{k^2-16\cos^2\theta\,(3-4\cos^2\theta)}}{4\cos\theta},
\end{align*}
so that $|y_1(x)|\ge|y_2(x)|$.

\begin{lemma}
\label{lem:x}
If $k\ge3$ then $\Delta_k(x)\ge0$, so that both $y_1(x)$ and $y_2(x)$ are real.

If $0\le k<3$ then $y_1(x)$ and $y_2(x)$ are complex conjugate to each other for
$$
\frac{3-\sqrt{9-k^2}}8<\cos^2\theta<\frac{3+\sqrt{9-k^2}}8,
$$
so that
$|y_1(x)|=|y_2(x)|=|3-4\cos^2\theta|^{1/2}$ in this case.
Furthermore, $|y_1(x)|=|y_2(x)|>1$ if and only if
\begin{alignat*}{2}
\frac{3-\sqrt{9-k^2}}8&<\cos^2\theta<\frac12 &\quad&\text{for} \;\; 0\le k<2\sqrt2, \\
\frac{3-\sqrt{9-k^2}}8&<\cos^2\theta<\frac{3+\sqrt{9-k^2}}8 &\quad&\text{for} \;\; 2\sqrt2\le k<3.
\end{alignat*}
\end{lemma}

\begin{proof}
Note that $16\cos^2\theta\,(3-4\cos^2\theta)\le\max_{0\le c\le1}16c(3-4c)=9$, hence
$$
\Delta_k(x)=k^2-16\cos^2\theta\,(3-4\cos^2\theta)\ge0
\qquad\text{if}\quad k\ge3.
$$
The second part of the statement is a mere computation.
\end{proof}

\begin{lemma}
\label{lem:y1}
If $k\ge2\sqrt2$ then $|y_1(x)|\ge1$ for all $x\in\mathbb C:|x|=1$.
\end{lemma}

\begin{proof}
Denote $c=\cos^2\theta$ for $x=\exp(i\theta)$, so that our task is to show that
\begin{equation}
|k+\sqrt{k^2-48c+64c^2}|\ge4\sqrt c
\label{ver1}
\end{equation}
for $0\le c\le1$. If $k^2-48c+64c^2\ge0$, meaning that either $k\ge3$ and $c\in[0,1]$ or
$2\sqrt2\le k<3$ and $c\in[0,(3-\sqrt{9-k^2})/8]\cup[(3+\sqrt{9-k^2})/8,1]$,
the inequality \eqref{ver1} is equivalent to
$$
\sqrt{k^2-48c+64c^2}\ge4\sqrt c-k.
$$
The latter inequality holds automatically when the right-hand side is nonpositive, that is, when $c\le k^2/16$.
If $c>k^2/16\ge1/2$ then
$$
\sqrt c(1-c)\le\frac k4\biggl(1-\frac{k^2}{16}\biggr)<\frac k4\cdot\frac12=\frac k8
$$
implying that $k^2-48c+64c^2<(4\sqrt c-k)^2=k^2-8k\sqrt c+16c$, and the required inequality follows.

If $k^2-48c+64c^2<0$ then $|y_1(x)|=|y_2(x)|=|y_1(x)y_2(x)|^{1/2}$ and
$$
|k+\sqrt{k^2-48c+64c^2}|
=|3-4c|^{1/2}.
$$
The latter expression is $\ge1$ whenever $0\le c\le1/2$; this indeed holds true for
$(3-\sqrt{9-k^2})/8<c<(3+\sqrt{9-k^2})/8$ since $2\sqrt2\le k\le 3$ in this case.

The required inequality \eqref{ver1} is thus established.
\end{proof}

\begin{lemma}
\label{lem:y2}
If $k\ge16/(3\sqrt3)=3.0792\dots$ then $|y_2(x)|\le1$ for all $x\in\mathbb C:|x|=1$.
\end{lemma}

\begin{proof}
To verify that $k-\sqrt{k^2-48c+64c^2}\le4\sqrt c$, equivalently
\begin{equation}
\sqrt{k^2-48c+64c^2}\ge k-4\sqrt c
\label{ver2}
\end{equation}
for $0\le c\le1$, we first notice that the inequality is trivially true for $c\ge k^2/16$ since
the right-hand side is then nonpositive. If $c<k^2/16$, the inequality \eqref{ver2} after squaring becomes equivalent
to $8\sqrt c(1-c)\le k$. The latter inequality holds true because the maximum of $\sqrt c(1-c)$ is attained at $c=1/3$
and is equal to $2/(3\sqrt3)$.
\end{proof}

\begin{proposition}
\label{prop3a}
If $k\ge16/(3\sqrt3)$ then
\begin{equation}
\frac{\d r(k)}{\d k}
=\frac1\pi\int_0^1\frac{\d c}{\sqrt{c(1-c)(k^2-48c+64c^2)}}.
\label{ell-1}
\end{equation}
\end{proposition}

\begin{proof}
Using the two lemmas above we conclude that for values of $k\ge16/(3\sqrt3)$ Jensen's formula gives us
\begin{align*}
r(k)
&=\m(R_k(x,y))=\m(\wt R_k(x,y))
=\frac1{2\pi i}\int_{|x|=1}\log|y_1(x)|\,\frac{\d x}x
\\
&=\Re\biggl(\frac1{2\pi i}\int_{|x|=1}\log\frac{k+\sqrt{k^2+4(x+x^{-1})(x^3+x^{-3})}}2\,\frac{\d x}x\biggr)
-\m(x+x^{-1})
\\
&=\frac1{2\pi}\,\Re\int_{-\pi}^\pi\log\frac{k+\sqrt{k^2-16\cos^2\theta\,(3-4\cos^2\theta)}}2\,\d\theta
\\
&=\frac2\pi\,\Re\int_0^{\pi/2}\log\frac{k+\sqrt{k^2-16\cos^2\theta\,(3-4\cos^2\theta)}}2\,\d\theta
\\
&=\frac2\pi\,\Re\int_0^1\log\frac{k+\sqrt{k^2-16t^2(3-4t^2)}}2\,\frac{\d t}{\sqrt{1-t^2}}
\end{align*}
which in turn implies that
\begin{align*}
\frac{\d r(k)}{\d k}
&=\frac2\pi\,\Re\int_0^1\frac1{\sqrt{k^2-16t^2(3-4t^2)}}\,\frac{\d t}{\sqrt{1-t^2}}
\\
&=\frac2\pi\int_0^1\frac1{\sqrt{k^2-16t^2(3-4t^2)}}\,\frac{\d t}{\sqrt{1-t^2}}.
\end{align*}
It remains to perform the change $c=t^2$.
\end{proof}

If $0<k<16/(3\sqrt3)$ then the cubic polynomial $f(t)=8t^3-8t+k$ has two real zeroes on the interval $0<t<1$, since
$f(0)=f(1)=k>0$ and $f(1/\sqrt3)=k-16/(3\sqrt3)<0$. Denote them $t_1(k)<t_2(k)$.

\begin{lemma}
\label{lem:y3}
If $|\cos\theta|=|x+x^{-1}|/2=t_1(k)$ then $|y_2(x)|=1$ for $k\le16/(3\sqrt3)$.

If $|\cos\theta|=|x+x^{-1}|/2=t_2(k)$ then
$$
|y_1(x)|=1 \quad\text{for}\;\; 0<k\le2\sqrt2 \qquad\text{and}\qquad
|y_2(x)|=1 \quad\text{for}\;\; 2\sqrt2\le k\le16/(3\sqrt3).
$$
\end{lemma}

\begin{proof}
Note that for the values of $x$ corresponding to $t_1(k)$ and $t_2(k)$ we always have $\Delta_k(x)\ge0$, so
that both $y_1(x)$ and $y_2(x)$ are real. The solutions of $|y_1(x)|=1$ and $|y_2(x)|=1$ correspond to solving
$$
k\pm\sqrt{k^2-16t^2(3-4t^2)}=4t,
$$
where $t=|\cos\theta|=|x+x^{-1}|/2$. By elementary manipulations the latter equation reduces to $8t^3-8t+k=0$,
and the remaining task is to distinguish whether we get $|y_1(x)|=1$ or $|y_2(x)|=1$. We do not reproduce this
technical but elementary analysis here.
\end{proof}

\begin{proposition}
\label{prop3b}
If $0<k<16/(3\sqrt3)$ then
\begin{equation}
\frac{\d r(k)}{\d k}
=\frac1\pi\biggl(\int_0^{t_1(k)^2}+\int_{t_2(k)^2}^1\biggr)\frac{\d c}{\sqrt{c(1-c)(k^2-48c+64c^2)}},
\label{ell-2}
\end{equation}
where $t_1(k)$ and $t_2(k)$, $0<t_1(k)<1/\sqrt3<t_2(k)<1$, are the real zeroes of the polynomial $8t^3-8t+k$.
\end{proposition}

\begin{proof}
To each $x$ on the unit circle we assign the real parameter $\theta$ such that $x=e^{i\theta}$ and real parameter $t=|x+x^{-1}|/2=|\cos\theta|\in[0,1]$.
The analysis of Lemmas~\ref{lem:x} to~\ref{lem:y3} shows that the ranges of $t$ that correspond to $|y_1(x)|\ge1$ and $|y_2(x)|\ge1$
are as follows: if $0<k<2\sqrt2$ then
$$
|y_1(x)|\ge1 \quad\text{for}\;\; t\in[0,1/\sqrt2]\cup[t_2(k),1]
\qquad\text{and}\qquad
|y_2(x)|\ge1 \quad\text{for}\;\; t\in[t_1(k),1/\sqrt2];
$$
and if $2\sqrt2\le k<16/(3\sqrt3)$ then
$$
|y_1(x)|\ge1 \quad\text{for}\;\; t\in[0,1]
\qquad\text{and}\qquad
|y_2(x)|\ge1 \quad\text{for}\;\; t\in[t_1(k),t_2(k)].
$$
Therefore,
\begin{align*}
r(k)
&=\frac1{2\pi i}\int_{|x|=1}\log\max\{|y_1(x)|,1\}\,\frac{\d x}x
+\frac1{2\pi i}\int_{|x|=1}\log\max\{|y_2(x)|,1\}\,\frac{\d x}x
\\
&=\frac2\pi\,\Re\biggl(\int_0^{1/\sqrt2}+\int_{t_2(k)}^1\biggr)\log\frac{k+\sqrt{k^2-16t^2(3-4t^2)}}{4t}\,\frac{\d t}{\sqrt{1-t^2}}
\\ &\qquad
+\frac2\pi\,\Re\int_{t_1(k)}^{1/\sqrt2}\log\frac{k-\sqrt{k^2-16t^2(3-4t^2)}}{4t}\,\frac{\d t}{\sqrt{1-t^2}}
\\ \intertext{if $0<k<2\sqrt2$ and}
&=\frac2\pi\,\Re\int_0^1\log\frac{k+\sqrt{k^2-16t^2(3-4t^2)}}{4t}\,\frac{\d t}{\sqrt{1-t^2}}
\\ &\qquad
+\frac2\pi\,\Re\int_{t_1(k)}^{t_2(k)}\log\frac{k-\sqrt{k^2-16t^2(3-4t^2)}}{4t}\,\frac{\d t}{\sqrt{1-t^2}}
\end{align*}
if $2\sqrt2\le k<16/(3\sqrt3)$.
Differentiating $r(k)$ we obtain
\begin{align*}
\frac{\d r(k)}{\d k}
&=\frac2\pi\,\Re\biggl(\int_0^{1/\sqrt2}+\int_{t_2(k)}^1-\int_{t_1(k)}^{1/\sqrt2}\biggr)\frac1{\sqrt{k^2-16t^2(3-4t^2)}}\,\frac{\d t}{\sqrt{1-t^2}}
\\ \intertext{if $0<k<2\sqrt2$ and}
&=\frac2\pi\,\Re\biggl(\int_0^1-\int_{t_1(k)}^{t_2(k)}\biggr)\frac1{\sqrt{k^2-16t^2(3-4t^2)}}\,\frac{\d t}{\sqrt{1-t^2}}
\end{align*}
if $2\sqrt2\le k<16/(3\sqrt3)$; here we have observed that the additionally occurring integrals in the process of differentiating
vanish because $\Re\log y_j(x)=\log|y_j(x)|=0$ by Lemma~\ref{lem:y3} in the corresponding cases.

Note that for both $0<k<2\sqrt2$ and $2\sqrt2\le k<16/(3\sqrt3)$ the result is the same:
$$
\frac{\d r(k)}{\d k}
=\frac2\pi\,\Re\biggl(\int_0^{t_1(k)}+\int_{t_2(k)}^1\biggr)\frac1{\sqrt{k^2-16t^2(3-4t^2)}}\,\frac{\d t}{\sqrt{1-t^2}}.
$$
To complete the proof we apply the substitution $t^2=c$.
\end{proof}

\begin{remark}
\label{rem2}
The integral in~\eqref{ell-1} is elliptic, while the integrals in~\eqref{ell-2} are incomplete elliptic:
the ``completion'' of the integrals will require integrating along $c\in(0,(3-\sqrt{9-k^2})/8)\cup((3+\sqrt{9-k^2})/8,1)$ if $0<k<3$
or $c\in(0,1)$ if $3\le k<16/(3\sqrt3)$ rather than along $c\in(0,t_1(k)^2)\cup(t_2(k)^2,1)$.
The incompleteness serves as a reason for the Mahler measure $r(k)$ not to be rationally related to $L'(E_k,0)$ for $|k|<16/(3\sqrt3)$.
\end{remark}

\begin{proposition}
\label{equ}
For $k$ positive real, $k\ne3$,
\begin{equation}
\int_0^1\frac{\d c}{\sqrt{c(1-c)(64c^2-48c+k^2)}}
=\int_{-12}^{-k(k-\sqrt{16+k^2})/2}\frac{\d v}{\sqrt{-(v+12)(v^2+k^2v-4k^2)}}.
\label{landen}
\end{equation}
\end{proposition}

\begin{proof}
Applying the substitution
$$
c=\frac{k(1+t)}{k+\sqrt{k^2+16}+(k-\sqrt{k^2+16})t}
$$
to the integral on the left-hand side we obtain
\begin{align*}
&
\int_0^1\frac{\d c}{\sqrt{c(1-c)(64c^2-48c+k^2)}}
\\ &\quad
=\sqrt2\int_{-1}^1\frac{\d t}{\sqrt{(1-t^2)(k^2-24+k\sqrt{k^2+16}+(-k^2+24+k\sqrt{k^2+16})t^2)}}
\\ &\quad
=2\sqrt2\int_0^1\frac{\d t}{\sqrt{(1-t^2)(k^2-24+k\sqrt{k^2+16}+(-k^2+24+k\sqrt{k^2+16})t^2)}}
\\ \intertext{(after the change $u=t^2$)}
&\quad
=\sqrt2\int_0^1\frac{\d u}{\sqrt{u(1-u)(k^2-24+k\sqrt{k^2+16}+(-k^2+24+k\sqrt{k^2+16})u)}}.
\end{align*}
Now the substitution
$$
u=\frac{2(v+12)}{-k^2+24+k\sqrt{k^2+16}}
$$
into the latter integral results in the the right-hand side in~\eqref{landen}.
\end{proof}

\begin{remark}
\label{rem3}
For $k>0$, $k\ne3$, the identity in Proposition~\ref{equ} relates the periods of the elliptic curves $E_k$ in~\eqref{Ek}
(which is isomorphic to $u^2=(v+12)(v^2+k^2v-4k^2)$) and
$$
\wh E_k:d^2=c(1-c)(64c^2-48c+k^2).
$$
The curves $E_k$ and $\wh E_k$ are not isomorphic but the latter one happens to be a quadratic twist of the former.
\end{remark}

\begin{proof}[Proof of Theorem~\textup{\ref{th2}}]
The equality of elliptic integrals in \eqref{landen} means that the derivatives of $p(k)$ and $r(k)$ coincide for $k\ge16/(3\sqrt3)$.
Thus $p(k)=r(k)+C$ for the range of~$k$, and the asymptotics \eqref{asymp} implies that $C=0$ and finishes the proof of the theorem.
\end{proof}

\section{Accurateness of Theorem~\ref{th2} and related comments}
\label{sfinale}

Though our Remarks \ref{rem1} and \ref{rem2} are aimed at explaining the choice of ranges for $k$ in Theorems~\ref{th1} and~\ref{th2},
in conclusion we would like to specifically address the difference between $\m(P_3)$ and $\m(R_3)$. The choice $k=3$ corresponds to a simultaneous
degeneration in the families of curves $P_k(x,y)=0$ and $R_k(x,y)=0$.

The curve
$$
P_3(x,y)=(x^2+x+1)y^2+3x(x+1)y+x(x^2+x+1)=0
$$
has genus 1; it is isomorphic to the conductor 15 elliptic curve $y^2+xy+y=x^3+x^2$
which has Cremona label \texttt{15a8} \cite[\href{http://www.lmfdb.org/EllipticCurve/Q/15/a/7}{Curve 15.a7}]{LMFDB}.
The proof of the evaluation
\begin{equation}
\m(P_3)=\frac16L'(\chi_{-15},-1)=0.99905183\dots
\label{mP3}
\end{equation}
was given in \cite[Example~3]{BRV03} (by two different methods!).

On the other hand,
$$
R_3(x,y)=(x+y-1)(x^2-xy+y^2+x+y)
$$
so that
$$
\m(R_3)=\m(x+y-1)+\m(x^2-xy+y^2+x+y)=L'(\chi_{-3},-1)+\m(x^2-xy+y^2+x+y).
$$
Following the technology and notation in \cite{BRV03} to compute the Mahler measure of $A(x,y)=x^2-xy+y^2+x+y$,
we first fix the rational parametrisation
$$
x=\frac{t-2}{t^2-t+1}, \qquad y=\frac{-t-1}{t^2-t+1},
$$
and compute the resultant of $A(x,y)$ and $A^*(x,y)=x^2y^2A(1/x,1/y)$:
$$
\operatorname{Res}_y(A,A^*)=3x^2(x^4+x^3-x^2+x+1).
$$
The quartic polynomial has exactly two complex conjugate zeroes
$$
x_1=\frac{3+i\sqrt{5+2\sqrt{13}}}{1+\sqrt{13}}
$$
and $x_1^{-1}$ of absolute value~1. The corresponding values of $y$ satisfying $|y|=1$ and $A(x,y)=0$ are
$y=y_1=x_1^{-1}$ for $x=x_1$ and $y=x_1$ for $x=x_1^{-1}$. The pair $(x_1,y_1)$ is generated by
$$
t_1=\frac{1-i\sqrt{5+2\sqrt{13}}}2.
$$
Note that in this case
\begin{align*}
\eta(x,y)
&=\eta\biggl(\frac{t-2}{t^2-t+1},\frac{-t-1}{t^2-t+1}\biggr)
\\
&=\d D\biggl(-\biggl[\frac{t+1}3\biggr]+2\biggl[\frac{t+1}{\zeta_6+1}\biggr]+2\biggl[\frac{t+1}{\zeta_6^{-1}+1}\biggr]\biggr),
\end{align*}
where the $1$-form $\eta(g,h)=\log|g|\,\d\arg h-\log|h|\,\d\arg g$ is attached to rational nonconstant functions $g$ and $h$ and
$$
D(z)=\operatorname{Im}\sum_{n=1}^\infty\frac{z^n}{n^2}+\arg(1-z)\,\log|z|
$$
denotes the Bloch--Wigner dilogarithm. Then by results in \cite{BRV03} the Mahler measure of $A(x,y)$ is equal to
$$
\m(A)
=\frac1\pi\biggl(D\biggl(\frac{t_1+1}3\biggr)-2D\biggl(\frac{t_1+1}{\zeta_6+1}\biggr)-2D\biggl(\frac{t_1+1}{\zeta_6^{-1}+1}\biggr)\biggr)
=0.68844794\dots\,.
$$
The resulting measure $\m(R_3)=1.01151388\dots$ visually appears to be different from~\eqref{mP3} confirming that $\m(P_k)\ne\m(R_k)$ at least for $k=3$.
Furthermore, $\m(R_3)$ does not seem to be a $\mathbb Q$-linear combination of $L'(\chi_{-3},-1)$ and $L'(\chi_{-15},-1)$.

\medskip
It would be interesting to establish the expected evaluation $\m(R_4)=-\frac13L'(E_{224\text{a}},0)$, hence also
for $\m(P_4)$ and $\m(Q_6)$, by using the recent formula of Brunault \cite{Bru15} for the regulator of Siegel units.
Note that the elliptic curve $R_4(x,y)=0$ does not possess a modular-unit parametrisation (so that the formula from \cite{Zud14} is not applicable)
and it is isomorphic to the curve $y^2=x^3+x^2-8x-8$ which has Cremona label \texttt{224a2}
\cite[\href{http://www.lmfdb.org/EllipticCurve/Q/224/a/1}{Curve 224.a1}]{LMFDB}.

Another related conjecture of Boyd \cite[Eq.~(3-12)]{Boy98} states that
$$
\m(Q_{-1})=\frac13L'(\chi_{-7},-1)+\frac16L'(\chi_{-15},-1)
=\frac{7\sqrt7}{12\pi}L(\chi_{-7},2)+\frac{5\sqrt{15}}{8\pi}L(\chi_{-15},2).
$$
Here $Q_{-1}(x,y)=0$ is an elliptic curve of conductor $210=2\times3\times5\times7$,
which is isomorphic to $y^2+xy=x^3+x^2-3x-3$ with Cremona label \texttt{210d1} \cite[\href{http://www.lmfdb.org/EllipticCurve/Q/210/a/3}{Curve 210.a3}]{LMFDB}.
Numerics indicates the lack of a modular-unit parametrisation in this case, though a suitable
parametrisation by Siegel units and the principal result from \cite{Bru15} are expected to confirm Boyd's observation for $\m(Q_{-1})$.

\medskip
\textbf{Acknowledgements.}
We owe our gratitude to Fran\c cois Brunault for helpful comments and discussions on a preliminary version of the note.

A part of the work was done during the authors' visit in the Centre de recherches math\'ematiques, Universit\'e de Montr\'eal,
and the second author's stay at the Max-Planck-Institut f\"ur Mathematik, Bonn, in Spring 2015.
We thank the staff of the institutes for the excellent conditions we experienced when conducting this research.

The second author acknowledges the support of the Australian Research Council and the Max-Planck-Institut f\"ur Mathematik.


\end{document}